\documentclass{amsart}
\usepackage[utf8]{inputenc}
\usepackage{amsmath}
 \usepackage{amsfonts}
 \usepackage{stmaryrd}
\usepackage{amssymb,latexsym}
\usepackage{amsfonts,mathrsfs}
\usepackage[margin=1.8in]{geometry}
\usepackage[all,cmtip]{xy}
\usepackage[colorlinks= true,pdfborder={0 0 0}]{hyperref}
\usepackage{graphicx}
 \usepackage{stmaryrd}
 \usepackage{enumitem}
 \usepackage{esint}
\usepackage{color}
\usepackage{xcolor}
\usepackage{comment}
\def\f#1#2{\frac{#1}{#2}}

\def\pa{\partial}
\def\n{\nabla}

\def\a{\alpha}
\def\b{\beta}
\def\ga{\gamma}

\def\({\left (}
\def\){\right )}
\def\<{\langle}
\def\>{\rangle}
\def\H{\mathbb{H}^n}

\newcommand{\bel}[1]{\begin{equation}\label{#1}}
\newcommand{\lab}[1]{\label{#1}}
\newcommand{\be}{\begin{equation}}
  \newcommand{\beq}{\begin{equation}}

\newcommand{\ba}{\begin{eqnarray}}
\newcommand{\ea}{\end{eqnarray}}
\newcommand{\rf}[1]{(\ref{#1})}

\newcommand{\qe}{\end{equation}}
\newcommand{\eeq}{\end{equation}}

\newtheorem{thm}{Theorem}[section]
\newtheorem{cor}[thm]{Corollary}
\newtheorem{lem}[thm]{Lemma}
\newtheorem{prop}[thm]{Proposition}
\newtheorem{defn}[thm]{Definition}
\newtheorem{rem}[thm]{Remark}
\newtheorem{claim}{Claim}[section]
\newtheorem*{acknowledgement*}{Acknowledgement}
\newtheorem{eg}{Example}[section]

\newcommand{\norm}[1]{\left\Vert#1\right\Vert}
\newcommand{\abs}[1]{\left\vert#1\right\vert}
\newcommand{\set}[1]{\left\{#1\right\}}
\newcommand{\real}{\mathbb R^n}

\newcommand{\R}{\mathbb R}
\newcommand{\eps}{\varepsilon}

\newcommand{\cta}{\theta}
\newcommand{\Om}{\Omega}
\newcommand{\om}{\omega}

\newcommand{\ho}{H\"ormander}

\newcommand{\half}{\frac{1}{2}}

\newcommand{\foo}[1]{\footnote{#1}}

\title{Lipschitz regularity of sub-elliptic harmonic maps into cat(0) space}
\author{Renan Assimos, Yaoting Gui, Jürgen Jost}
\address{Renan Assimos, Faculty of Mathematics and Physics, Leibniz Universit\"at Hannover. }
\address{Yaoting Gui, Beijing International Center for Mathematical Research, Peking University. }
\address{Jürgen Jost, Max Planck Institute for Mathematics in the Sciences, Leipzig.}
\email{renan.assimos@math.uni-hannover.de, ytgui@bicmr.pku.edu.cn, jjost@mis.mpg.de}

\begin{document}
\begin{abstract}
  We prove the local Lipschitz continuity of sub-elliptic harmonic maps between certain singular spaces, more specifically from the
  $n$-dimensional Heisenberg group into $CAT(0)$ spaces. Our main theorem establishes that these maps have the desired Lipschitz regularity, extending the Hölder regularity in this setting proven by Gui et. al in \cite{gui2022subelliptic} and obtaining same regularity as in H-C. Zhang and X-P. Zhu \cite{zhang2018lipschitz} for certain sub-Riemannian geometries, see also \cite{mondino2022lipschitz} and \cite{gigli2022regularity} for generalisations to RCD spaces. The present result paves the way for a general
  regularity theory of sub-elliptic harmonic maps, providing a versatile approach applicable beyond the Heisenberg group.
\end{abstract}

\maketitle
\numberwithin{equation}{section}
\section{\bf Introduction}

Sub-Riemannian manifolds are a class of non-trivial Riemannian like geometries that have found profound applications in mathematics and physics. A sub-Riemannian space is represented as a pair $(M, H)$, where $M$ is a connected smooth manifold, $H$ is a sub-bundle of the tangent bundle $TM$ equipped with a smooth metric $g_H$, and $H$ satisfies the  Hörmander condition, meaning that despite being only a sub-bundle of $TM$, the vector fields therein and their brackets still generate the whole tangent bundle. In particular, a Laplace type operator can be defined, which is hypo-elliptic and hence still possesses regularity properties similar to ordinary Laplace operators. Developing further the fundamental contributions of Hörmander \cite{hormander1967hypoelliptic}, a general theory for such linear partial differential equations (PDEs) has emerged. Variational problems, however, often lead to non-linear systems. Here, we study a model problem and 
	 explore non-linear sytems of hypo-elliptic partial differential equations  concerning harmonic maps defined over a sub-Riemannian manifold and mapping into a locally compact metric space in the sense of Alexandrov, or $CAT(0)$. 
	
	Harmonic maps into singular spaces first occurred in the work of  Gromov and Schoen \cite{gromov1992harmonic} who studied such maps into Euclidean Bruhat-Tits buildings. These consist of branching flat spaces and include  metric trees.  Because of that branching, they have non-positive curvature in the sense of Alexandrov. The general theory of harmonic maps into such spaces started with the works of Jost \cite{jost1994equilibrium} and  Korevaar and Schoen \cite{korevaar1993sobolev}. Since at  that time, there was no general notion of differentiability of maps into such spaces, the energy $E(u)=\int_M \|du\|^2 dvol_M(x)$ for a mapping $u:M\to Y$ between smooth Riemannian manifolds ($du$ is the differential of $u$, and $ \|du\|^2$, where the norm comes from the Riemannian metrics, is called the energy density) has to be replaced by an expression involving only distances in the target, like
        \begin{equation}
          \label{1}
  E_\epsilon(u)=\frac{1}{c_\epsilon}\int_M
  \int_{B(x,\epsilon)}d_Y^2(u(x),u(y))\ d\operatorname{vol}(y)\ d\operatorname{vol}(x)
        \end{equation}
  where $d_Y(.,.)$ is the distance function of the target $Y$ and  $B(x,\epsilon)$ is a ball in the domain $M$ with center $x$ and
radius $\epsilon$. In the expression above, $c_\epsilon$ is a
normalization constant that is chosen so that in smooth situations we
have
\begin{equation}\label{2}
\lim_{\epsilon \to 0}E_\epsilon(u)=E(u).
\end{equation}
(We assume here, in accordance with the setting of this paper that $M$ is still a smooth manifold, although generalizations to more general domains are natural and possible.)
 Korevaar and Schoen \cite{korevaar1993sobolev}  defined such quantities as measures, and this was further investigated in \cite{guo2021harmonic}. This will also be useful in our context.  We describe it here. Consider two metric spaces $\left(M, d_M\right)$ and $\left(Y, d_Y\right)$, with $\Omega \subset M$ a domain equipped with a Radon volume measure, denoted by vol (note that in our context of sub-Riemannian geometry, such a volume measure in general is  not canonically defined, an issue that will be addressed below). For a given  $\epsilon>0$, along with a Borel measurable map $u: \Omega \rightarrow Y$, an approximating energy functional $E_{\epsilon}^u$ is defined on $C_0(\Omega)$ by
	
	$$
	E_{\epsilon}^u(\phi):=c(n) \int_{\Omega} \phi(x) \int_{B_x(\epsilon) \cap \Omega} \frac{d_Y^2(u(x), u(y))}{\epsilon^{n+2}} d \operatorname{vol}(y) d \operatorname{vol}(x).
	$$
        where $n$ is the dimension of $M$. 
	The approximate energy functional $E_{\eps}^u$ can be shown to converge to an energy functional $E^u$ as $\eps\rightarrow0$ in an almost monotonic way, and hence defines an energy for a general map. This is the case, for example for the Heisenberg group.
	When the domain is a sub-Riemannian manifold, minimizing maps of the limiting energy functional $E^u(\phi)$ are called sub-elliptic harmonic maps.
	
	As in the classical case, the regularity theory for such harmonic maps from a domain in an Alexandrov space with curvature bounded from below into a complete $CAT(0)$, i.e., a length space of non-positive curvature in the sense of Alexandrov,  
	is an intriguing question.  J. Jost \cite{jost1997generalized} and F.H. Lin \cite{lin1997analysis} established the Hölder regularity for such maps.
	In a seminal paper,  H-C. Zhang and X-P. Zhu\cite{zhang2018lipschitz} then  proved the Lipschitz regularity for Alexandrov spaces with generalised sectional curvature lower bounds. The Lipschitz continuity of harmonic maps plays a central role in establishing rigidity theorems within the realm of geometric group theory. Recently, the Lipschitz regularity of the harmonic maps has been extended to  RCD spaces by A. Mondino and D. Semola \cite{mondino2022lipschitz}, and shortly after also by N.Gigli \cite{gigli2022regularity}. We thus have  a full answer to the regulairty problem of  harmonic maps between singular spaces when the target has non-positive Alexandrov curvature.
	
	In the sub-Riemannian world, there is one extra difficulty. Since the metric is only defined on the sub-bundle $H$, there is no canonical volume measure on $M$. The theory of sub-elliptic harmonic maps started with the work of J. Jost and C-J. Xu \cite{jost1998subelliptic}. To overcome this issue, they simply take an arbitrary Riemannian metric $g$ and endow the bundle complementary to $H$ with the restriction of $g$, then this bundle is declared to be orthogonal to $H$. This procedure yields a measure and a volume form, albeit it is important to note that the choice of the metric $g$ is arbitrary and does not relate to the underlying sub-Riemannian structure.
	
	Despite the initial difficulty of not having a canonical way to associate a volume measure to the sub-Riemannian setting, it was shown in the paper J. Jost- C-J. Xu \cite{jost1998subelliptic} that many crucial results from quasi-linear and nonlinear elliptic PDEs can be extended to the hypo-elliptic case, demonstrating the profound interplay between sub-Riemannian geometry and PDEs. More specifically, they obtained an existence and regularity theorem for these sub-elliptic harmonic maps defined on a domain of $\mathbb{R}^n$ taking values into a Riemannian manifold under a Dirichlet condition on the domain and a convexity condition on the target. This convexity assumption appeared first in the work  of S. Hildebrandt, H. Kaul and K-O. Widman \cite{hildebrandt1977existence} on the images of classical harmonic maps into smooth targets without non-positive curvature hypothesis.
	
	A natural question now is  the regularity of sub-elliptic harmonic maps defined on sub-Riemannian manifolds when the target is a singular $CAT(0)$ space. The Hölder continuity of such maps has been recently obtained by Y.T.Gui et al.  \cite{gui2022subelliptic} and their main theorem is as follows:
	
	\begin{thm}
		Let $u:M\rightarrow N$ be a harmonic map from a sub-Riemannian manifold $M$
		into a locally compact metric $CAT(0)$ space $N$. Then $u$ is H\"older continuous.
	\end{thm}
	
	What about the Lipschitz regularity? In the present paper we develop an approach to establish the Lipschitz regularity of hypo-elliptic harmonic maps in the setting described above. Our main result asserts that 
	\begin{thm}\label{mainthm}
		Let $u:\H\rightarrow N$ be an energy minimizing harmonic map, where $\H$ is the n dimensional Heisenberg group and $N$ is a CAT(0) space. Then $u$ is locally Lipschitz continuous.  
	\end{thm}
	
	The Heisenberg group serves as an Euclidean-like model for the sub-Riemannian setting, which makes Theorem~\ref{mainthm} a first step into the regularity theory of sub-elliptic harmonic maps in such general targets, but we would like to emphasise that the present approach only requires the sub-Riemannian manifold to be the Heisenberg group in a technical way in Lemma~\rf{lem5.3}. More specifically, in order to prove this Lemma we need a particular symmetry assumption which is satisfied, for example, by the Heisenberg group.
	
	The paper is organised as follows, in section 2, we present some necessary preliminaries. In section 3, we will review the sub-elliptic harmonic function and illustrate the basic idea with a smooth target. Section 4 is devoted to the establishment of the Sobolev space of the sub-Riemannian manifold with values in singular spaces. Section 5 consists of basic calculus on the new Sobolev space. and the establishment of the main theorem and its proof.

    The results of this paper constitute the main part of the second author's PhD thesis.

    \begin{acknowledgement*}
    The first two authors would like to thank the Max Planck Institute for Mathematics in the Sciences for its hospitality.
    \end{acknowledgement*}

\section{\bf Preliminaries}
\setcounter{equation}{0}
In this section, we introduce some basic notation and definitions.
\subsection{Geometry of sub-Riemannian manifolds}

Let $M$ be a smooth manifold and $S$ be a distribution of the tangent bundle $TM$ with rank $m$. A triple $(M, Q_S,S)$ is called a sub-Riemannian manifold of rank $m$, if $Q_S$ is a Riemannian metric on $S$. Suppose $X_i$ is a local basis of $S$ defined on some open subset $U\subset M$. Denote
  $$\Gamma(U,S)=\text{span}\{X_1,\cdots,X_m\},$$
and inductively define
\[
\Gamma^{j+1}(U,S)=\Gamma^j(U,S)+[\Gamma(U,S),\Gamma^j(U,S)],
\]
where $[\cdot,\cdot]$ is the Lie bracket of the vector fields.
\begin{defn}
S is said to satisfy the H\"{o}rmander condition of order $r$ if
\[
\Gamma^r(U,S)_x=T_xM\quad \forall x\in  U.
\]
\end{defn}
  
It is easy to check that the definition is independent of the particular choice of the local basis. 

In order to define a Laplace operator, in addition, one needs a volume form. However, such a volume form is not canonically defined. One can introduce a compatible volume form as follows. Choose any background Riemannian metric $g_0$, consider the vertical distribution $V$ of $S$ with respect to the metric $g_0$, to  obtain an orthogonal decomposition of the tangent bundle, $TM=S\oplus V$. Then we define a metric $g=g_S+g_{0,V}$, i.e. we assert that the vectors in $V$ and $S$ be orthogonal and use the restriction metric on $V$ and $S$ respectively. In this way, we get an extension of the original sub-Riemannian metric and we can define a volume form. Classical examples are given by the Carnot group, in particular the Heisenberg group, or  contact manifolds,
see \cite{dong2020eells}.\medskip

We now introduce the fundamental distance function that is naturally associated with the distribution, called the Carnot-Carath\'{e}odory distance (cc distance for short). We emphasize that this distance is not induced by the Riemannian metric $g$, and hence it may be  very different from the Riemannian distance. \medskip

We say a curve $\gamma :I\rightarrow M$ is horizontal if $\dot{\gamma}(t)\in S_{\gamma(t)}, \forall t\in I$. We call $(\gamma(t),\xi(t))$ a lift of the curve $\gamma(t)$ if $g_S(\gamma(t))\xi(t)=\dot{\gamma}(t)$. Define the length of the curve $\gamma$ by
\[
l(\gamma)=\int_{a}^{b}\sqrt{\< g_S(\xi(t)),\xi(t)\>}dt.
\]

For any two points $p,q\in M$, the distance between them is then defined as the infimum of all
horizontal curves connecting the two points $p,q$. More precisely, we define
\[
d_{cc}(p,q)=\inf_{\gamma}\{l(\gamma):\gamma \;\;\text{is a horizontal curve}\}.
\]
By a classical theorem of Chow and Rashevsky\cite{chow2002systeme,rashevsky1938any}, there always exists a horizontal curve on a connected manifold, hence the distance is finite. We refer to the book \cite{burago2001course} for a proof of
the fact that $d_{cc}$ is indeed a metric on $M$.

Fix a base point $q_0\in M$ and choose a local orthonormal frame $X_j, i=1,\cdots, k$, for our distribution $S$. Let $\Phi$ be the corresponding local flows. We use them to move in the horizontal directions, 
\[
\Phi_i(t)(q)=q+tX_i(q)+O(t^2),
\]
for $t$ small. Note that the length of such a simple curve with $0<t<\eps$ is simply $\eps$. We may move in the $\Gamma^2\setminus \Gamma$ directions along horizontal paths by applying the commutators $\Phi_{ij}(t) = [\Phi_i(t),\Phi_j(t)]$ to the point $q_0$. Now $\Phi_{ij}(t)(q_0)=q_0+t^2[X_i,X_j](q_0)$, so if we restrict to $0<t<\eps$ we will move by a Euclidean amount $\eps^2$ in the $\Gamma^2/\Gamma$
direction. We continue the process of taking commutators and brackets until we have 
exhausted the tangent space. 

For multi-indices $I=(i_1,i_2\cdots,i_m),1\leq i_j\leq k$, define vector fields $X_I$ inductively by 
	$X_I=[X_{i_1},X_J]$, where $J=(j_2,j_3,\cdots,j_m)$. By the H\"ormander assumption we can select a local frame for the entire tangent bundle from amongst the $X_I$. We choose such a frame and relabel it $Y_i, i=1,\cdots,n$, to respect the canonical filtration: $\set{Y_1=X_1,\cdots,Y_k=X_k}$ span $S$ near $q_0$; $\set{Y_1,\cdots,Y_{n_2}}$ span $\Gamma^2$ near $q_0$ and so on, where $(k,n_2,n_3,\cdots,n_r)$ is called the growth vector of the distribution at $q_0$. For each chosen $Y_i$ of the form $X_I$, let $\om_i$ be the length $\abs{I}$. Thus $\abs{\om}_i=m$ if and only if $Y_i\in\Gamma^m$ and $Y_i\neq\Gamma^{m-1}$. The assignment $i\mapsto\om_i$ is called 
	the weighting associated to the growth vector. 
	The $\om$-weighted 
	box of size $\eps$ is the point set 
	\[
	B^\om(\eps)=\set{y\in\real:\abs{y}_i\leq\eps^{\om_i},i=1,\cdots,n}.
	\]
\begin{thm}[Ball-box theorem]
There exist positive constants $c<C$ and $\eps_0>0$ such that for all $\eps<\eps_0$, such that the metric ball $B_\eps(q_0)$ defined by the sub distance looks like
\[
B(c\eps)\subset B^\om_\eps(q_0)\subset B(C\eps). 
\]
\end{thm}
\begin{rem}
This theorem in particular asserts that the topological properties of  sub-Riemannian geometry are similar to those of Riemannian geometry. However, the Hausdorff dimension, which depends heavily on the metric is quite different, see below.
\end{rem}

We next recall some necessary analysis on sub-Riemannian geometry. Let $\Om\subset M$ be a bounded domain. Choose a local coordinate chart $x_i$ on $\Omega$, we write $X_j=b^{jk}\partial_{x^k}$; the adjoint of $X_j$ then
	is defined by $X_j^*f=-\partial_{x^k}(b^{jk}f)$ for $f\in C_0^1(\tilde\Om)$.
	\begin{defn}[\ho{} operator]
		The \ho{} operator is defined
		\be
		H=\sum_{j=1}^mX_j^*X_j=-\sum_{i,k=1}^m\partial_{x^i}( a^{ik}(x)\partial_{x^k}),
		\qe
		where $a^{ik}(x)=b^{ji}(x)b^{jk}(x)$. It can be seen as the subelliptic version of
		the Laplace operator. 
	\end{defn}

\begin{rem}
		In particular,
		$a^{ik}\in C^\infty(\tilde\Om)$
		and $(a^{ik}(x))_{i,k=1,\dots , m}$ is symmetric and positive semidefinite.
		It need not be positive definite, however, and therefore, in general,
		$H$ is not elliptic.
	\end{rem}

\begin{eg}[The Heisenberg group $\H$]\label{Heisenbergexample}
The following example plays a significant technical role later on this work, so we investigate it thoroughly. On $\mathbb{R}^{2n+1}$, define the following group operation, 
\[
(x,y,t)\cdot(u,v,s)=(x+u,y+v,t+s+\f12(xv-yu)),
\]
where $x,y,u,v\in\real$ and $xv$ is the standard inner product in $\real$. The resulting group is denoted by $\H$ and called the Heisenberg group. One may define a sub-Riemannian structure on $\H$ with the horizontal vectors given by 
\[
X_i=\pa_{x_i}-\f12y_i\pa_t,Y_i=\pa_{y_i}+\f12x_i\pa_t, Z=\pa_t.
\]
It is easy to see that 
\[
Z=[X_i,Y_i], 0=[X_i,Z]=[Y_i,Z]=[X_i,Y_j].
\]
The Lie algebra of the Heisenberg group is given by
$\mathfrak{g}=\mathfrak{g}_1\oplus\mathfrak{g}_2$, where $\mathfrak{g}_i=span\set{X_i,Y_i}$. Due to the simple graded feature of this Lie algebra, one can define a natural dilation and translation operators $\delta$ and $\tau$, respectively, on $\mathfrak{g}$. Namely for any $\eps>0$, let $w=(x,y,t)$, and $v=(x',y',t')$. Define 
\[
\delta_\eps(w)=(\eps x,\eps y, \eps^2t),\quad \tau_w(v)=w\cdot v.
\]
Now one uses the exponential map to induce a dilation and a translation operators on the group itself. We abuse notation and still call the induced maps on $\H$ by $\delta_w$ and $\tau_w$. The measures behaviour under dilation is also good, meaning it satisfies 
\[
\mathcal{L}^{2n+1}(\delta_\eps(E))=\eps^{2n+2}\mathcal{L}^{2n+1}(E).
\]
The cc-distance is well behaved under the translation $\tau_w$ and the dilation $\delta_\eps$. Moreover, it is left-invariant:
\[
d_{cc}(\tau_w v,\tau_w v')=d_{cc}(v,v'),
\]
and
\[
d_{cc}(\delta_\eps w,\delta_\eps v)= d_{cc}(w,v)
\]
\end{eg}
\begin{rem}
		Since $\H$ is a Lie group, there exists a natural left translation operator, hence the unit ball under the non-isotropic metric is indeed translation invariant. In other words, this space is locally homogeneous. Hence, we may just consider the unit ball centred at the origin. The scaling property also indicates why the Hausdorff dimension become 2n+2. It has two dimension in the $Z$-direction.
	\end{rem}
	\begin{rem}
		The Heisenberg group $\H$ satisfies the so-called Measure Contraction Property(MCP(0,2n+3)), which is verified in \cite{juillet2009geometric}. This property in particular enables us to define the Sobolev space along the lines of Korevaar and Schoen, \cite{korevaar1993sobolev}.
	\end{rem}
	\begin{defn}
		The Pansu differential of a smooth function in the direction $w$ can be naturally given as in the Euclidean space, 
		\[
		(Pf)_p(w)=\lim_{\eps\rightarrow0}\f{f(p\cdot\delta_\eps(w))-f(p)}\eps,
		\]
		if the limit is exists.
	\end{defn}
	A celebrated theorem analogous to the Lebesgue differential theorem asserts that \cite{tan2004sobolev}
	\begin{thm}
		For a smooth function $f\in\H$, there holds
		\[
		\lim_{\eps\rightarrow0}\int_{\rho(w)\leq1}\abs{(Pf)_p(w)-\f{f(p\cdot\delta_\eps(w))-f(p)}\eps}^2d\sigma(w)=0.
		\]
	\end{thm}
	This theorem will be useful when we explicitly calculate the differential of a map from the Heisenberg group.

\section{\bf Analysis on sub-Riemannian manifolds: smooth targets}
In this section, we introduce some analytic tools and define the Sobolev space of maps from sub-Riemannian into smooth manifolds. Below we always assume $\Om\subset M$ is fixed unless otherwise stated. 
\subsection{Energy functional}
Let $N$ be a complete Riemannian manifold without boundary
of dimension bigger or equal to $2$. Consider maps
$$f:(\bar\Om,d_\eps)\to( N,d_N).$$
Since $(\bar\Om,d_\eps)$, $(N,d_N)$ are Riemannian manifolds
with metric tensors $(a_{ij}^\eps)$ and $(g_{\alpha\beta})$ respectively, if $f:\bar\Om\to N$ is a $C^1$ map, we can define the energy density
\ba
\nonumber
e_\eps(f)
&=& \half( a^{ij}(x)+\eps\delta_{ij}) g_{\alpha\beta}(f)
    \frac{\partial f^\alpha}{\partial x_i}\frac{\partial f^\beta}{\partial x_j}\\
&=& \half g_{\alpha\beta}(f)X_jf^\alpha X_jf^\beta+
    \frac{\eps}{2}g_{\alpha\beta}(f)
    \frac{\partial f^\alpha}{\partial x_j}\frac{\partial f^\beta}{\partial x_j}
\ea
Then the energy of $f$ is simply
\beq
E_\eps(f)=\int_\Om e_\eps(f)d\Om
\eeq
with$d\Om=dx$ being the Lebesgue measure. For $f\in H^1(\Om,N)$,
$$\lim_{\eps\to 0}E_\eps(f)=\int_\Om e(f)dx,$$
where
\beq
e(f)=\half g_{\alpha\beta}(f)X_jf^\alpha X_jf^\beta.
\eeq

So we define the energy of a map
$f:(\bar\Om,d_{cc})\to(N,d_N)$ by
\beq
E(f)=\int_\Om e(f)dx.
\eeq
The Euler-Lagrange equation for this energy functional can be written in the form
\bel{eqn3.5}
Hu^\alpha+\Gamma_{\beta\gamma}^\alpha(u)X_ju^\beta X_ju^\gamma=0,\quad
1\leq\alpha\leq\nu,
\qe
that is, as a harmonic map type system with the H\"ormander operator taking the role of the Laplacian. The equation \rf{eqn3.5} is equivalent to the following, for every $1\leq\a\leq\nu$, we have
\[
-a^{ij}\f{\pa^2u^\a}{\pa x_i\pa x_j}-\pa_ja^{ij}\pa_iu^\a+\Gamma_{\b\ga}^\a(u)a^{ij}
\pa_iu^\b\pa_ju^\ga=0.
\]

The basic idea to derive the Lipschitz regularity is to establish a Bochner formula for the energy density function. 
 
\subsection{The basic strategy} 
 To illustrate the idea, we may first look at a baby case of a sub-elliptic harmonic function. Recall that to obtain a suitable regularity, we may derive a differential inequality for the finite difference of the original function. So let $u:\Om\rightarrow\R$ be a sub-elliptic harmonic function, fix any constant vector $w$ with small norm. Set $u_w(x)=u(\delta_w(x))$, consider the function $u_\eta(x)=(1-\eta)u+\eta u_w$, where $\eta$ is a smooth function with compact support in $\Om$. In this case, since both of $u$ and $u_w$ are harmonic, we have\foo{Here by a sub-elliptic harmonic function, we mean a function minimising the horizontal energy $\int_\Om\abs{Xu}^2=\int_\Om\sum_j\abs{X_ju}^2$ with the same boundary values.}
\bel{inq3.6}
\int_\Om \abs{Xu}^2\leq \int_\Om\abs{Xu_\eta}^2,
\qe
and
\bel{inq3.7}
\int_{\Om_w} \abs{Xu_w}^2\leq \int_{\Om_w}\abs{Xu_{1-\eta}}^2,
\qe
where $\Om_w=\{x\in\Om| d(x,\pa\Om)>\abs{w}\}$. We can expand the right-hand side of the two inequalities \rf{inq3.6} and \rf{inq3.7}, add them together, and cancel the 0-th order term to get \foo{The domain of the two integrals in \rf{inq3.6} and \rf{inq3.7} are different, but due to the cut-off function, after cancelling the 0-th order term, we can choose $w$ with sufficiently small norm so that $\Om$ and $\Om_w$ coincide.}
\begin{align}
0 & \leq\int-X\eta\cdot X(u-u_w)^2-2\int\eta\abs{X(u-u_w)}^2\notag\\
& + \int\abs{X\eta}^2(u-u_w)^2+2\int\eta^2\abs{X(u-u_w)}^2+\int X(\eta^2)\cdot X(u-u_w)^2.
\end{align}
Now, replacing the test function $\eta$ by $t\eta$ and letting $t\rightarrow0$, we finally arrive at the inequality 
\bel{inq3.9}
0\leq\int-X_j\eta X_j(u-u_w)^2-2\int\eta\abs{X(u-u_w)}^2.
\qe
Equation~\rf{inq3.9} is nothing but the weak form of the following differential inequality 
\bel{inq1.15}
H(u-u_w)^2+2\abs{X(u-u_w)}^2\leq 0.
\qe
So combined with the Sobolev embedding inequalities, we can proceed by the De Giorgi-Nash-Moser iteration to control the value at the centre of a ball by its average of the integral over such a ball.

\section{\bf Analysis on sub-Riemannian manifolds: singular targets}
In this section, we start by introducing some basic facts on nonpositively curved spaces. Then we turn our attention to the definition of the Sobolev space of maps into singular targets, where we borrow heavily some arguments from N. Korevaar and R. Schoen~\cite{korevaar1993sobolev}. Finally, we illustrate how to use the non-positivity of the curvature to derive a B\"ochner identity. We let $B_r(q)=\set{y\in M:d_{cc}(y,q)<r}$ denote the metric ball centred at $q$ with radius $r$ in $M$, while $(N,d)$ denotes a CAT(0) space. 
\subsection{Basics on nonpositively curved space}
Since there is no longer a linear structure on our target space $N$, we need to use geodesics to interpolate points between two maps, see \cite{jost1997generalized}. To be more precise, fix a smooth function $\eta$ with compact support and $0\leq \eta\leq1$. Let $u_0,u_1:\Om\rightarrow N$ be two maps. For $x,y\in\Om$, set 
\[
\eta_-=\min(\eta(x),\eta(y)).
\]
Let $\ga_x:[0,1]\rightarrow N$ be the arc-length geodesic with 
\[
\ga_x(0)=u_0(x)\quad\text{and}\quad \ga_x(1)=u_1(x).
\]
We also set $u_\eta(x)=\ga_x(\eta(x))$. Since the subsequent considerations are symmetric in $x$ and $y$, we may assume $\eta_-=\eta(y)$ without loss of generality. In this setting, following the ideas in \cite{jost1997generalized} we consider the ordered quaternary sequence 
\[
\{u_{\eta_-}(y), u_{\eta_-}(x), u_{1-\eta_-}(x), u_{1-\eta_-}(y)\}.
\]

Here we need to use a quadrilateral comparison property in CAT(0) spaces, and we recall this fact and useful equivalent propositions. See \cite[Theorem~2.1.2, Corollary~2.1.3]{korevaar1993sobolev}
\begin{thm}\lab{thm4.1}
Let $(N,d)$ be a CAT(0) space. Fix any quaternary sequence $\{P,Q,R,S\}\subset N$ and the corresponding sequence $\{\bar{P},\bar{Q}, \bar{R}, \bar{S}\}\subset\mathbb{R}^2$. Define geodesics $P_t=\ga_{PS}(t)$ and $Q_s=\ga_{QR}(s)$ using the common convention that $\gamma_{PS}(0)= P$ and $\gamma_{PS}(1)=S$, and similar for $\gamma_{QR}$. The corresponding Euclidean points are
\[
\bar{P}_t=(1-t)\bar{P}+t\bar{S}, \quad\text{and}\quad \bar{Q}_s=(1-s)\bar{Q}+s\bar{R}.
\]
Then we have 
\[
d(P_t,Q_s)\leq \abs{\bar{P}_t-\bar{Q}_s}.
\]
\end{thm}
\begin{cor}
Under the assumptions of Theorem \ref{thm4.1}, we have that for any $0\leq t,s\leq 1$
\begin{equation}\label{inq4.1}
    \begin{array}{ccl}
        d^2(P_t,Q_t) & \leq & (1-t)d_{PQ}^2 +td_{RS}^2 \\
        & &\\
        & & -t(1-t)\Big(s(d_{SP}-d_{QR})^2+(1-s)(d_{RS}-d_{PQ})^2\Big).
    \end{array}
\end{equation}

\begin{equation}\label{inq4.2}
    \begin{array}{rcl}
    d^2(P,Q_t)+d^2(S,Q_{1-t}) & \leq &  d_{PQ}^2+d^2_{RS}+t(d_{SP}^2-d_{QR}^2)+2t^2d_{QR}^2\\
    & &\\
     - t\Big(s(d_{SP}-d_{QR})^2&+&(1-s)(d_{RS}-d_{PQ})^2\Big).     
    \end{array}
\end{equation}
\end{cor}

Now applying inequality \rf{inq4.2} to the quaternary sequence 
\[
\{u_{\eta_-}(y), u_{\eta_-}(x), u_{1-\eta_-}(x), u_{1-\eta_-}(y)\},
\]
with $t=(\eta(x)-\eta(y))/(1-2\eta(y))$ and set $f(x)=d^2(u_0(x),u_1(x))$, we get 
\begin{equation}\lab{inq4.3}
    \begin{array}{rcl}
        &&d(u_\eta(x),u_\eta(y))^2+d(u_{1-\eta}(x),u_{1-\eta}(y))^2\leq d(u_{\eta_-}(x),u_{\eta_-}(y))^2\\
        &&\\
        &&+d(u_{1-\eta_-}(x),u_{1-\eta_-}(y))^2 - (\eta(y)-\eta(x))(1-2\eta(y))(f(y)-f(x))\\
        &&\\
         &&+ 2(f(x)+f(y))\Big(\f{\eta(y)-\eta(x)}{1-2\eta(y)}\Big)^2.
    \end{array}
\end{equation}
Moreover, inequality \rf{inq4.1} implies that
\begin{equation}\lab{inq4.4}
\begin{array}{rcl}
d(u_{\eta_-}(x),u_{\eta_-}(y))^2+d(u_{1-\eta_-}(x),u_{1-\eta_-}(y))^2 &\leq& d(u_0(x),u_0(y))^2\\
&\\
&+&d(u_1(x),u_1(y))^2.
\end{array}
\end{equation}
Combining the inequalities \rf{inq4.3} and \rf{inq4.4}, we finally obtain 
\begin{equation}\lab{inq4.5}
\begin{array}{rl}
      &d(u_\eta(x),u_\eta(y))^2+d(u_{1-\eta}(x),u_{1-\eta}(y))^2\leq d(u_0(x),u_0(y))^2\\
     &\\
         &+d(u_1(x),u_1(y))^2-  (\eta(y)-\eta(x))(1-2\eta(y))(f(y)-f(x)).
\end{array}    
\end{equation}

\subsection{The approximate energy functional}
The inequality \rf{inq4.5} enable us to use some approximate energy measure to study the regularity of the sub-elliptic harmonic maps, see \cite{korevaar1993sobolev}. We state roughly idea: For any $\eps>0$, consider the following approximate energy density function, 
\[
e_\eps(x)=\f1{\mu(B_\eps(x))}\int_{B_\eps(x)}\f{d^2(u(x),u(y))}{\eps^2}d\mu(y).
\]

Define $\Om_\eps=\{x\in \Om|d_{cc}(x,\pa\Om)\geq\eps)\}$.

\begin{claim}\lab{clm4.1}
The function $e_\eps$ is integrable on $\Om_\eps$.
\end{claim}
\begin{proof}
The integrability of $e_\epsilon$ follows from the fact that the volume of the metric ball in $\Omega$ is of polynomial growth with a constant depending on a chosen compact set, but not on the particular choice of the centre, see \cite{nagel1985balls}. Precisely, let $K\Subset \Om$ be chosen so that $\Om_\eps\subset K$, then there exist two positive constants such that
\[
C_1(K)\Lambda(x,\eps)\leq\mu(B_\eps(x))\leq C_2(K)\Lambda(x,\eps),
\]
where $\Lambda(x,\eps)$ is a polynomial in the variable $\eps$ with smooth coefficients bounded away from zero. In particular, 
\[
\int_{\Om_\eps} \f1{\Lambda(x,\eps)}d\mu(x)<\infty.
\]
\end{proof}

\subsection{A key lemma}
We define a family of measures $d\mu_\eps=e_\eps(x)dx$ and the corresponding approximate energy functional 
\[
E_\eps(f)=\int_\Om f(x)d\mu_\eps(x) \quad \forall f\in C_0(\Om).
\]
We want to show that the approximate energy functional converges to some certain functional, which will be called the energy functional. The convergence depends on the following fact, called the sub-partition lemma, which is the analogue of the fact that refining a partition of a curve will increase the length. For any $\eps>0$ sufficiently small, let $f\in C_0(\Om)$, and $\omega(f,2\eps)$ be the oscillation of $f$ over the metric ball $B_{2\eps}(x)$.

\begin{lem}[Sub-partition Lemma]\label{suplem}
There exists a constant $C>0$, such that 
\be
E_\eps(f)\leq\sum_i\lambda_i(E_{\lambda_i\eps}(f_\eps^C)),
\qe
where $\sum_i\lambda_i=1, \lambda_i>0$ and $f_\eps^C=(1+C\eps)(f+\omega(f,2\eps))$.
\end{lem}
\medskip
For the proof of this lemma, we will recall a measure property due to K-T. Sturm \cite{sturm1998diffusion}; more applications of this property have also appeared in the works K. Kuwae~\cite{kuwae2001generalized} and S. Ohta~\cite{ohta2007measure}, and references therein. For any $x,y\in\Om$, let $\Phi_t(x,y), t\in[0,1]$ be the minimal horizontal geodesic connecting $x,y$, which is measurable as a two variables map \footnote{The choice of the geodesic is however not unique, but we may pick one of them such that the resulting map is measurable.}. 
\begin{defn}
A metric measure space $(\Om,\rho,\mu)$ is said to satisfy the Measure Contraction Property with exceptional set $Z$ if and only if there is a closed set $Z\subset\Om$ with $\mu(Z)=0$, such that for each compact set $Y\subset \Om\setminus Z$, there exist numbers $R>0, \cta<\infty$ and measurable
maps $\Phi : X \times X \rightarrow X$ for all $t\in[0,1]$ with the following property: \medskip

\begin{enumerate}[label=(\roman*)]    
    \item  For $\mu$-a.e. $x,y\in Y$ with $d_{cc}(x,y)<R$ and all $s, t\in [ 0, 1] $,
            \[
            \Phi_{0}(\:x,\:y)\:=\:x,\quad\Phi_{t}(\:x,\:y)\:=\:\Phi_{1-\:t}(\:y,\:x)\:\]
            \[
            \Phi_s(x,\Phi_t(x,y))=\Phi_{st}(x,y),
            \]
            $$
            d(\Phi_s(x,y),\Phi_t(x,y))\leq\cta|s-t|d(x,y);
            $$

    \item  For all $r<R$, $x\in Y$, all measurable subsets $A\subset B_r(x)\cap Y$ and         all $t\in[0,1]$, there holds
            \bel{inq4.7}
                \f{\mu(A)}{\mu(B_r(x))}\leq \cta\f{\mu(\Phi_t(x,A))}{\mu(B_{tr}(x))}.
            \qe
\end{enumerate}
\medskip
We say that $(\Om,\rho,\mu)$ satisfies the strong Measure Contraction Property
with exceptional set $Z$ if and only if for each compact $Y\subset \Om\setminus Z$, the constants $\cta$ can be chosen arbitrarily close to 1 and if the measure $\mu$ satisfies the doubling condition locally. If the above conditions are satisfied with $Z=\emptyset$, then we say that $(\Om,\rho,\mu)$
satisfies the (weak or strong, respectively) Measure Contraction Property without exceptional set.
\end{defn}
\begin{rem}
Assume there is a constant $C>0$ such that for all $r<R$, for $x\in Y\, a.e.$, and $y\in\Om$ with $\rho(x,y) < \cta \cdot r$, we have 
\bel{inq4.8}
\f1C\leq\f{\mu(B_r(x))}{\mu(B_r(y))}\leq C.
\qe
Then inequality \rf{inq4.7} is equivalent to the condition \rf{inq4.8}. 
\end{rem}
\begin{rem}
The map $\Phi_t$ obviously maps the ball $B_r(x)$ into the ball $B_{tr}(x)$. Roughly speaking the condition gives a control of the distortion of the volume of subsets under the map $\Phi_t$ by the inequality \rf{inq4.8}.
\end{rem}
\begin{rem}
There are some examples that do satisfy the Measure Contraction Property. K-T. Sturm gives an example which fits in our setting \cite{sturm1998diffusion}. Namely, he considers the triple $(\mathbb{R}^2,d_a,dx)$ with a given symmetric matrix
\[a(x)=
\begin{pmatrix}
1 & 0 \\ 0 & \phi^2(x)
\end{pmatrix}
\]
where $\phi$ is a smooth positive function with only discrete zeros $Z_1$. Set $Z=\R\times Z_1$, then $a$ is invertible on $\mathbb{R}^2\setminus Z$ and $\mathcal{L}^2(Z)=0$. Defining the length of a curve $\gamma$ by $l(\gamma)=\int_0^1\sqrt{\dot{\gamma}a^{-1}\dot{\gamma}}dt$, one can introduce the metric 
\[
d_a(x,y)=\inf\{l(\gamma)|\gamma \:\text{is a horizontal curve connecting}\; x \,\text{to}\, y\}.
\]
Endowed with this metric, it is shown that the space satisfies the strong MCP with exceptional set $Z$. In general, we may consider higher dimensional analogues.
\end{rem}
\begin{rem}
The Heisenberg group is another example satisfying the strong MCP without exceptional set \cite{juillet2009geometric}. There are actually more examples, L. Rizzi proved this property for carnot groups \cite{rizzi2016measure}. Recently, the MCP property was proved to be valid in analytic sub-Riemannian manifolds \cite{badreddine2020measure}
\end{rem}
\begin{proof}[Proof of Lemma \ref{suplem}]
  Employing the Measure Contraction Property alongside a standard formula for chamge of variables, we adapt the proof of \cite[Lemma~1.3.1]{korevaar1993sobolev} to our setting. 
  For a point $x$ within the domain of $f$ and a small $\varepsilon$, we select a geodesic connecting $x$ to $y$ such that $|x-y| = \varepsilon$. The choice of the geodesic is arbitrary, as any will suffice. Denote $\psi: [0,1] \to \Omega$ as this chosen geodesic trajectory with uniform velocity from $x$ to $y$. Construct a sub-partition as follows:
  \[
  \begin{aligned}
    z_0 &= x, \\
    z_i &= \psi\left(\sum_{k=1}^i \mu_k\right), \quad \text{for } i=1, \ldots, m.
  \end{aligned}
  \]
  Consequently,
  $$
  |z_i - z_{i-1}| = \mu_i \varepsilon,
  $$
  and by using the triangle inequality, we have
  $$
  d(u(x), u(y)) \leq \sum_{i=1}^m d(u(z_{i-1}), u(z_i)).
  $$
  Applying the generalised $L^2$-triangle inequality for $(fe_\varepsilon)$, we get
  \[   
    E_\varepsilon(f)^{\frac{1}{2}} \leq \sum_{i=1}^m \mu_i \left( \underset{|z_i - z_{i-1}| \leq \mu_i \varepsilon}{\operatorname*{\int\int}} f(x) e_{\mu_i\varepsilon}(z_{i-1}, z_i) \, d\sigma_{\varepsilon}(x, y) \right)^{\frac{1}{2}}.
  \]
  For the integral representation, we define
  \[
  e_\varepsilon(x,y) = \frac{d^2(u(x),u(y))}{\varepsilon^2},
  \]
  and the measure $d\sigma_{\varepsilon, x}(y)$ over the set $B(x,\varepsilon) = \{ y \mid d_{cc}(x,y) < \varepsilon \}$ as
  \[
  d\sigma_{\varepsilon, x}(y) = \frac{d\mu_y}{\sqrt{\mu(B_r(x))}}.
  \]
  Applying the MCP, we transform variables from $(x,y)$ to $(z_{i-1}, z_i)$
 in the above inequality. The integral transformation is expressed as
  \[
  d\sigma_{\varepsilon, x}(y) = \frac{d\sigma_{\varepsilon, x}(y)}{d\sigma_{\mu_i \varepsilon, z_{i-1}}(z_i)} d\sigma_{\mu_i \varepsilon, z_{i-1}}(z_i).
  \]
  The MCP then gives the bound
  $$
  \frac{d\sigma_{\varepsilon, x}(y)}{d\sigma_{\mu_i \varepsilon, z_{i-1}}(z_i)} \leq \cta^4(1 + D\varepsilon).
  $$
  We refer to \cite{sturm1998diffusion} for further details on the steps involving the MCP property.
\end{proof}
With these examples in hand and the proof of the sub-partition lemma in \cite{sturm1998diffusion}, we may conclude that on a sub-Riemannian space satisfying the MCP, we can define a suitable energy functional, and then an energy measure with the help of a Riesz representation given by 
\[
E(f)=\lim_{\eps\rightarrow 0}E_\eps(f)=\int_\Om f\,de,\quad for \; f\in C_0(\Om).
\]
We define the energy of an $L^2(\Om,N)$ map $u$ by 
\bel{inq1.24}
E^u=\sup_{\begin{subarray}{} 0\le f\le 1\\
f\in C_0(\Om)\end{subarray}}\lim_{\eps\rightarrow0}E_\eps(f).
\qe
\begin{defn}(The Sobolev Space)
We say that a map $u:\Om\rightarrow N$ is in the Sobolev space $H^1(\Om,N)$, if $u\in L^2(\Om,N)$ and \rf{inq1.24} is finite, in this case, we define its $H^1$ norm by 
\[
\norm{u}_{H^1}=\sqrt{E}+\norm{u}_{L^2}.
\]
\end{defn}
Here it seems to arise an interesting question: Is the generalised curvature-dimension condition due to F. Baudoin \and N. Garofalo\cite{baudoin2016curvature} equivalent to the Measure Contracting Property?

\subsection{The directional energy}
Next we want to show that the energy measure $de$ obtained with the above method is indeed absolutely continuous with respect to the Lebesgue measure.
To see this, we need to study the approximate directional energy functional. At a first glance, we may look at the directional energy density $\f1\eps d(u(x),u(x+\eps X_j))$. However, in sub-Riemannian geometry $X_j$ may contain zeroes, so the flow is possibly trivial. Hence we look at unit vector fields contained in the horizontal subspace at each point. Along this direction, we introduce the following notation: Let $V(x)$ denote the subspace spanned by the horizontal vector fields, which we may call the horizontal subspace, that is, 
\[
V(x)=\text{span}\set{X_j(x)}.
\]
And also let $q(x)$ be the dimension of the subspace $V(x)$. Choose any constant vector $\om\in V(x)$. If $q(x)=q$ is a constant function (this is the case for example the Heisenber gruop), then any two vectors $\om_1\in V(x), \om_2\in V(y)$ differ only by a rotation $R\in SO(q)$, hence we introduce the following directional approximate energy density function 
\[
e_\eps^\om(x)=\f{d^2(u(x),u(x+\eps\om))}{\eps^2}. 
\]

\begin{rem}
Here we can only formulate the directional approximate energy function in the direction $\om\in V(x)$, because the distance function in the sub-Riemannian setting is non-isotropic, in other words, the distance is not equivalent to the Euclidean distance.
\end{rem}
The function $e_\eps^\om$ is well-defined, i.e. it is integrable, which can be shown in a similar way as in Claim \ref{clm4.1}; more than that, the volume of the metric ball is of polynomial growth. Therefore the approximate energy functional can be defined as follows.
\[
E_\eps^\om(f)=\int_\Om fe_\eps^\om dx, \quad f\in C_0(\Om),
\]
and use the subpartition lemma \ref{suplem} to derive a suitable directional energy functional, which we may denote by $E^\om$, that is, 
\[
E^\om (f)=\int_\Om f(x)de^\om.
\]
where $de^\om$ is the limit of the approximate energy density function.

\subsection{Energy density function}
Assume now that $u\in H^1(\Om,N)$. Firstly, we outline the steps to prove the absolute continuity of the approximate directional energy functional, where we take inspiration from several lemmata in \cite{korevaar1993sobolev}\footnote{The absolute continuity is also contained in \cite{guo2021harmonic}, but there is no point-wise convergence of the approximated energy density.}.

\vspace{0.3cm}

\noindent \textbf{Step 1} \textit{Study the map $u$ restricted to the integral flow $\gamma_x$ generated by some vector field $\om_j$ at a point $x$.} For any fixed vector field $\om_j$, let $\bar{x}_j(x,t)$ be the family of transformations generated by $\om_j$ at the point $x$, namely, 
\[
\dfrac{d}{dt}\bar{x}_j(x,t)|_{t=0}=\om_j(\bar{x}_j(x,t)),\quad \bar{x}_j(x,0)=x.
\]
 
Define the directional approximate energy density function 
\bel{inq1.25}
e_\eps^j(x)=\f{d^2(u(x),u(\bar{x}_j(x,t)))}{\eps^2}.
\qe
For simplicity, we omit the dependence of the vector $\om_j$, and denote the quantity by $e_\eps$, and also set $\ga_x(t)=\bar{x}_j(x,t)$. We need to show that $u|_{\gamma_x}$ is a $H^1(I_x,N)$-map for almost all $x\in\Om$.\foo{Here we need to clarify the exact meaning of a map restricted on an interval, this can be illustrated by the probability measure on the continuous curves. For example see \cite{heinonen2015sobolev}.} The goal then is to prove the absolute continuity of the directional energy density with respect to the Lebesgue measure. 
\begin{proof}[Proof of step 1]
We need to show that\foo{With a slight abuse of notation, the energy functional here is defined on a one-dimensional interval.}
\[
\sup_{\begin{subarray}{}0\leq f\leq 1 \\ f\in C_0(I_x)
\end{subarray}}
\limsup_{\eps\rightarrow0}E_\eps(f)=E<\infty.
\]
Let $f\geq0$ be a continuous function supported on $I_x$. We have
\begin{align*}
    \int_\Om E_\eps(f)d\mu = & \int_\Om \int_{I_x}f(t)\Big(\f{d^2(u(\ga_x(t+\eps)),u(\ga_x(t)))}{\eps^2}+
    \f{d^2(u(\ga_x(t-\eps)),u(\ga_x(t)))}{\eps^2}\Big)\\
    \leq & \int_\Om \int_{I_x}f(t)\Big(\f{d^2(u(\ga_y(\eps)),u(y))}{\eps^2}+\f{d^2(u(\ga_y(-\eps)),u(y))}{\eps^2}\Big)J\phi_t^{-1}dtd\mu\\
    \leq & C_t\int_{I_0}\int_\Om f \Big(\f{d^2(u(\ga_y(\eps)),u(y))}{\eps^2}+\f{d^2(u(\ga_y(-\eps)),u(y))}{\eps^2}\Big)d\mu dt\\
    \leq & C_t\abs{I_0} E+o(\eps),
\end{align*}
where $\phi_t(x)=\ga_x(t)$ and we assume $I_x$ is uniformly bounded, i.e. $I_x\subset I_0, \forall x\in \Om$. In the second line of the inequality, we used spherical average to approximate the energy functional and the following simple fact involving the change of variables $y=\ga_x(t)$.
\begin{lem}
Let $\phi(t,x)$ be the flow generated by the vector field $\om_j$,
that is, 
\[
\dfrac d{dt}\phi(t,x)=\om_j(\phi(t,x)), \quad \phi(0,x)=x.
\]
Then for each $t$, the family $\phi_t(x)=\phi(t,x)$ is a diffeomorphism from a neighbourhood of $x$ into itself. Moreover, the Jacobian $J\phi_t$ of the map $x\mapsto\phi_t(x)$,  and its inverse $J\phi_t^{-1}$ are bounded. 
\end{lem}
\end{proof}

\noindent\textbf{Step 2} \textit{Similar to \cite[Lemma~1.9.2]{korevaar1993sobolev}, we need to find a candidate of the energy density function on the interval $I_x$}. It is to be expected to study the $\R$-valued function 
\[
v(s)=d(u(t+s),u(t)),
\]
which is in the space $H^1(I_x-t,\R)$, hence H\"older continuous. Therefore, the Newton-Leibniz formula gives a desired measure $de_p$ satisfying $\abs{\n_sv}ds\leq de_p$. In particular, we have $\abs{v'(s)}ds\leq e_1(t)dt$.
\begin{proof}[Proof of step 2]
First we need to verify that the function $v$ is in the space $H^1(I_x)$. This can be seen as a corollary of the following fact. 
\begin{lem}
Let $v_1, v_2\in H^1(\Om,N)$ be two maps with finite energy. Let $h:N\times N\rightarrow\R$ be a Lipschitz function with Lipschitz constant $L$. If we define the function $f(x)=h(v_1(x),v_2(x))$, then $f\in H^1(\Om)$\foo{By the Sobolev space $H^1(\Om)$ we mean the function space consisting of $L^2$ functions with bounded energy $\int_\Om\abs{Xf}^2=\int_\Om\sum_j\abs{X_jf}^2<\infty$ in the weak sense.} and 
\[
\abs{Xf}^2\leq L^2(de^{v_1}+de^{v_2}),
\]
where $de^{v_i}$ is the energy measure of the map $v_i$.
\end{lem}
\begin{proof}
This follows almost directly from the triangle inequality. Indeed, we have
\begin{align*}
    \f{\abs{f(x)-f(y)}^2}{\eps^2} & = \f{\abs{h(v_1(x),v_2(x))-h(v_1(y),v_2(y))}^2}{\eps^2}\\
    & \leq L^2\f{d^2(v_1(x),v_1(y))+d^2(v_2(x),v_2(y))}{\eps^2}.
\end{align*}
Since $v_i$ has bounded energy, we conclude that $f$ also has bounded energy. 
\end{proof}
\begin{rem}\lab{rem4.13}
In particular, if we choose $v_2\equiv y_0$, and $h(y_1,y_2)=d(y_1,y_2)$, we may conclude that 
\[
\abs {Xd(v,y_0)}^2\leq Cde^v.
\]
\end{rem}
By Remark~\rf{rem4.13}, we have that  $v\in H^1(I_x)$, so that it is H\"older continuous by the Sobolev embedding inequality\foo{Note that we use the non-isotropic distance, but in the direction $X_j$, this distance is equivalent to the Euclidean one.}. In particular $v$ is absolutely continuous, and we have the following estimate due to Newton-Leibniz formula. 
\bel{inq4.11}
v(\eps)-v(0)\leq \int_0^\eps\abs{v'(s)}ds\leq \int_t^{t+\eps}e_1(r)dr,
\qe
where $e_1$ is the energy density function of the map $u|_{\ga_x}$\foo{A better notation for this energy density function would be $e_1(t,x)$. We use the simplied term $e_1$ if there is no confusion with the point and time we are considering.}.
\end{proof}
\begin{rem}
Here the density function $e_1$ comes from the following fact: For a map $u\in H^1(\Om,N)$, it holds naturally that $u\in W^{1,p}(\Om,N)$ for $1\leq p<2$. Moreover, the energy measure $de_p$ is indeed absolutely continuous with respect to the Lebesgue measure. This can be see easily checked as follows. 

Let $E$ be the total energy of the map $u$. For any Borel set $A\subset\Om$ and any $\delta>0$, we may choose some function $f\in C_0(\Om)$ with $f\equiv1$ on $A$, and $\abs{supp(f)}\leq \abs{A}+\delta$. Then we have 
\begin{align*}
    \int_\Om f(x)e_{\eps,p}d\mu = & \int_\Om f(x)\fint_{B_\eps(x)}\f{d^p(u(x),u(y))}{\eps^p}d\mu(y)d\mu\\
    \leq & C_{\Om,p}\int_\Om f(x)e_\eps(x)^{\f p2}d\mu. 
\end{align*}
Hence 
\[
e_p(A)\leq C_{\Om,p}E^{\f p2}(\abs{A}+\delta)^{\f{2-p}2}.
\]
This shows that the energy measure $e_p$ is absolutely continuous with respect to the Lebesgue measure.
\end{rem}

\vspace{0.3cm}

\noindent \textbf{Step 3} \textit{Show that there exists a point-wise limit 
\[
\f{d(u(t+\eps),u(t))}\eps\rightarrow e_1(t),\quad as\quad \eps\rightarrow0,
\]
for almost all $t\in I_x$}.
\begin{proof}[Proof of step 3]
We follow the proof of \cite[Lemma~1.9.3]{korevaar1993sobolev} with minimal adaptation to our setting. For simplicity, we set $f_\eps(t)=\f{d(u(t+\eps),u(t))}\eps$. It suffices to show 
\[
\lim_{\eps\rightarrow0}f_\eps(t)\geq e_1(t), \quad \text{and} \quad
\lim_{\eps\rightarrow0}f_\eps(t)\leq e_1(t).
\]
The second inequality follows from the inequality \rf{inq4.11}. In fact, dividing both sides of inequality \rf{inq4.11} by $\eps$ and letting $\eps\rightarrow0$, the desired result follows from the Lebesgue differentiable theorem\foo{This theorem relies on some covering property of the metric space and doubling property of the measure, of which both hold in our setting.}. To show the first inequality, we define for any $\delta>0$ the set 
\[
S_\delta=\{t\in I_x,\textit{$t$ is a Lebesgue point for $e_1$}, \text{and} \,
\liminf_{\eps\rightarrow0}f_\eps(t)\leq e_1(t)-\delta\}.
\]
We now want to show the set $S_\delta$ is negligible. For that, let $t\in S_\delta$. By definition, for any small number $\eps_0>0$, there exists some $\eps<\eps_0$ such that the intervals $(t-\eps,t+\eps)$ covers $S_\delta$ and for which 
\[
f_\eps(t)\leq e_1(t)-\delta, \quad \text{and} \quad e_1(t)-\f1\eps \int_t^{t+\eps}e_1(r)dr\leq\f\delta2.
\]
By the 5-times covering lemma, we may find a disjoint subfamily of this covering $\bar{I}_i=(t_i-\eps,t_i+\eps)$ with $\sum_i\abs{I_i}\geq \f1{10}\abs{S_\delta}$. We set $I_i=[t_i,t_i+\eps_i)$ and the complementary intervals are denoted by $J_k=[s_k,s_k+\eps_k)$. $P=\{I_i, J_k\}$ together consists of a partition of the interval $I_x$ with maximum length bounded by $\eps_0$. We have the following estimate.
\begin{align}\lab{inq4.12}
    \Sigma & := \sum_id(u(t_i+\eps_i), u(t_i))+\sum_kd(u(t_k+\eps_k), u(t_k))\notag\\
    & \leq \sum_i\Big(\int_{I_i}e_1(r)dr-\f\delta2 \eps_i\Big)+\sum_k\int_{J_k}e_1(r)dr\notag\\
    & \leq \int_{I_x}e_1(r)dr-\f\delta {20}\abs{S_\delta}.
\end{align}
Now we can show that the summation $\Sigma$ converges to the integral $\int_{I_x}e_1(r)dr$, and we can then conclude from the inequality \rf{inq4.12} that $\abs{S_\delta}=0$.In fact, since $de_1=e_1\,dt$ is absolutely continuous, we have 
\begin{align}
    E_1=\int_{I_x}e_1(t)dt & = \lim_{\eps\rightarrow0}\int_{I_x-\eps}\f{d(u(t),u(t+\eps))}\eps dt\notag\\
    & = \lim_{\eps\rightarrow0}\int_{I_{x,\eps}}\sum_{i=0}^{[\f1\eps]-2}\f{d(u(t+i\eps),u(t+(i+1)\eps))}\eps dt\notag\\
    & := \lim_{\eps\rightarrow0}\f1\eps\int \Sigma_\eps(t)dt,
\end{align}
where $\abs{I_{x,\eps}}=\eps$. Hence for any $\eps,\delta>0$, we may choose some $t_\eps\in I_{x,\eps}$ such that $S_\eps(t_\eps)>E_1-\delta$. Let $P=\set{t_l}_{l=1}^k$ be a partition of $I_x$, subject to $N\norm{P}<\eps$ for some large enough $N$. We pick some sub-partition $Q$ of $P$ for which $t_i\in(t_\eps+i\eps-\f\eps N,t_\eps+i\eps+\f\eps N)$, $i=1,\cdots,[\f1\eps]-2$. Of course, we have 
\[
\sum_i\abs{t_\eps+i\eps-t_i}\leq \f{\abs{I_x}}N.
\]
It follows that 
\begin{align*}
    \Sigma = \sum_Pd(u(t_j),u(t_{j+1})) & \geq \sum_{Q}d(u(t_i),u(t_{i+1}))\\
    & \geq S_\eps(t_\eps)-2\sum\abs{\int_{t_i}^{t_\eps+i\eps}e_1(t)dt}\\
    & \geq \int_{I_x}e_1(t)dt-\delta-2(\f{\abs{I_0}}N)^{\f{p-1}p}(E_1)^{\f1p}.
\end{align*}
\end{proof}

\vspace{0.3cm}

\noindent \textbf{Step 4} \textit{Increase the power $p=1$ in the above argument to the power $p>1$}. We need to show that the energy measure $de$ is absolutely continuous with the Lebesgue measure, that is, $de=e(x)d\mu(x)$, and the general formula 
\bel{inq1.28}
e(x)=e_1(x)^2, \quad for \quad a.e. x\in\Om.
\qe
Furthermore, there is a representative of $u$ with
\bel{inq4.15}
\f{d^2(u(t+\eps),u(t))}{\eps^2}\rightarrow e^2_1(t),\quad as\quad \eps\rightarrow0.
\qe
In particular, the point-wise convergence \rf{inq4.15} holds for $t=0$, namely, 
\[
\f {d^2(u(\ga_x(\eps)),u(x))}{\eps^2}\rightarrow e_1^2(x), \quad as \quad\eps\rightarrow0.
\]
\begin{proof}[Proof of step 4]
As before we follow \cite[Lemma~1.9.5]{korevaar1993sobolev} minimally adapted to our setting. We need to prove the inequalities 
\[
e_1^2d\mu\leq de, \quad \text{and} \quad 
e_1^2d\mu\geq de.
\]
The first inequality can be seen from \rf{inq4.15} together with the Fatou Lemma,
\begin{align*}
    \int_\Om f(x)e_1^2d\mu = & \int_\Om \liminf_{\eps\rightarrow0}f(x) \f{d^2(u(\ga_x(\eps)),u(x))}{\eps^2}d\mu \\
    \leq & \liminf_{\eps\rightarrow0}\int_\Om f(x)\f{d^2(u(\ga_x(\eps)),u(x))}{\eps^2}d\mu\\
    \leq & \int_\Om f(x)de.
\end{align*}
To see the other side of the inequality, we fix some $f\in C_0(\Om)$, then 
\[
\int_\Om fde=\lim_{\eps\rightarrow0} \int_\Om f\f{d^2(u(\ga_x(\eps)),u(x))}{\eps^2}d\mu.
\]
It follows that 
\[
\int_\Om fde=\int_\Om f\f{d^2(u(\ga_x(\eps)),u(x))}{\eps^2}d\mu+\delta_1(\eps).
\]
However, for fixed $\eps>0$, there also holds the equality
\[
\int_\Om f\f{d^2(u(\ga_x(\eps)),u(x))}{\eps^2}d\mu= \lim_{p\rightarrow 2}\int_\Om f\f{d^p(u(\ga_x(\eps)),u(x))}{\eps^p}d\mu.
\]
We will use the following notation for the above integral: 
$$
\int_\Om f\f{d^p(u(\ga_x(\eps)),u(x))}{\eps^p}d\mu=\delta_2(p)
$$
Therefore, we have 
\bel{inq1.30}
\int_\Om fde=\int_\Om f\f{d^p(u(\ga_x(\eps)),u(x))}{\eps^p}d\mu + \delta_1(\epsilon)=\delta_2(p)+\delta_1(\eps).
\qe
We finally recall the following estimate uniformly in $p$, the proof then comes from the same line as the sub-partition lemma, since by definition
\[
E_\eps(f)\leq \int_\Om f_\eps^Ce_1^pd\mu,
\]
and the dominated convergence theorem applies. Hence we obtain the desired result by first letting $p\rightarrow2$ and then $\eps\rightarrow0$.
\end{proof}
In conclusion, we have proven the following with steps 1-4
\begin{prop}
The directional energy density is absolutely continuous with respect to the Lebesgue measure, that is, 
\[
de^\om=\abs{u_*(\om)}^pd\mu,
\]
where $\abs{u_*(\om)}$ is the density function of the energy measure of power 1 in the direction $\om$.
\end{prop}
 
\section{\bf Calculus on the Sobolev Space}
In this section, we want to illustrate that the newly defined Sobolev space is a suitable candidate to carry usual variational analysis. For example, we have the lower semi-continuity, the B\"ochner identity and so on.  
\subsection{Lower semi-continuity}

We here verify the following fact. For target $N=\R$, we simply denote $H^1(\Omega)=H^1(\Omega,\R)$.
\begin{thm}\lab{thm5.1}
	For our new Sobolev space, there holds the semi-continuity property. Namely, let $f_i\in H^1(\Om,N)$ be a sequence of Sobolev functions with uniformly bounded Sobolev norm, then $f_i$ converges weakly to some Sobolev function $f_\infty\in H^1(\Om,N)$, and we have 
	\[
	de_\infty\leq \liminf_{i\rightarrow\infty}de_i,
	\]
	where $de_i$\foo{Distinguish this from our previous definition of the approximation directional energy density, which was denoted by $de^j$. } is the corresponding energy density functions. Moreover, the newly defined Sobolev space
	coincides with the classical one if we take $N=\R$. 
\end{thm}
\begin{proof}[Proof of Theorem \ref{thm5.1}]
	It suffices to verify the semi-continuity property. For fixed $f\in C_0(\Om)$, we have 
	\[
	E_\eps^i(f)\leq E^i(f_\eps^C)=E^i(f)+E^i(f_\eps^C-f).
	\]
	It follows that 
	\[
	E_\eps(f)\leq \liminf_{i\rightarrow\infty}(f_\eps^C)+E(C\eps\abs{f}+\om(f,2\eps)).
	\]
	Therefore the map $u$ has finite energy, and by letting $\eps\rightarrow0$, we get 
	\[
	E(f)\leq \liminf_{i\rightarrow\infty}E^i(f).
	\]
\end{proof}
\subsection{The B\"ochner identity in Heisenberg group}
We now pay our attention back to look at the inequality \rf{inq4.5}. As before, let $\phi_\eps(x)$ be the flow generated by the vector field $\om$. By taking $y=\phi_\eps(x)$, integrating the inequality \rf{inq4.5} against a nonnegative function $\psi\in C_0(\Om)$. In the sequel, for simplicity, we denote the directional approximate energy density by $e^\om_{u,\eps}$ with respect to the map $u$, we have
\begin{align}\lab{inq1.31}
	& \int_\Om \psi e^\om_{u_\eta,\eps} d\mu+\int_\Om \psi e^\om_{u_{1-\eta},\eps} d\mu
	\leq\int_\Om \psi e^\om_{u_0,\eps} d\mu+\int_\Om \psi e^\om_{u_1,\eps} d\mu\notag\\
	- & \int_\Om\f{\eta(\phi_\eps(x))-\eta(x)}\eps \f{f(\phi_\eps(x))-f(x)}\eps(1-2\eta(\phi_\eps(x)))d\mu
\end{align}
Recall $f(x)=d^2(u_0(x),u_1(x))$, which is in the space $H^1(\Om,\R)$ by Theorem \ref{thm5.1}. We claim that
\begin{claim}
	For any $\psi\in C_0(\Om)$, 
	\bel{}
	\lim_{\eps\rightarrow0}\int_\Om \psi\f{\eta(\phi_\eps(x))-\eta(x)}\eps 
	\f{f(\phi_\eps(x))-f(x)}\eps d\mu
	=\int_\Om\psi \eta_*(\om)f_*(\om)d\mu.
	\qe
	In other words, the measure $\f{\eta(\phi_\eps(x))-\eta(x)}\eps 
	\f{f(\phi_\eps(x))-f(x)}\eps d\mu$ converges weakly to the measure $\eta_*(\om)f_*(\om)d\mu$\foo{Here for a $H^1$ function the notion $\eta_*(X_j)$ coincides with the usual weak derivative $X_j\eta$, we  use the new notation just for consistency. Moreover, the Lipschitz constant is comparable with the norm of the weak derivative.}.
\end{claim}
\begin{proof}
	We first note the following simple equality
	\begin{align}\lab{inq4.19}
		& (\eta(y)+f(y)-\eta(x)-f(x))^2\notag \\
		= & (\eta(y)-\eta(x))^2+(f(y)-f(x))^2+2(\eta(y)-\eta(x))(f(y)-f(x)).
	\end{align}
	Thus, by taking $y=\phi_\eps(x)$ and dividing both side of \rf{inq4.19} by $\eps^2$, integrating the equality against a nonnegative function $\psi$ and letting $\eps\rightarrow0$, we get 
	\[
	\abs{(\eta+f)_*(\om)}^2=\abs{\eta_*(\om)}^2+\abs{f_*(\om)}^2+2\eta_*(\om)f_*(\om).
	\]
\end{proof}
If we use the approximating energy functional, we find that
\begin{align}\lab{inq4.20}
	& E_\eps^{u_\eta}(\psi)+E_\eps^{u_{1-\eta}}  \leq E_\eps^{u_0}(\psi)+E_\eps^{u_1}(\psi)\notag\\
	- & 
	\int_\Om \psi\f{\eta(\phi_\eps(x))-\eta(x)}\eps 
	\f{f(\phi_\eps(x))-f(x)}\eps d\mu.
\end{align}
Now set $\eps\rightarrow0$ in \rf{inq4.20} and use the absolutely continuity of the directional energy measure, we finally arrive at the following inequality, 
\begin{align}
	&    \abs{{u_\eta}_*(\om)}^2+\abs{{u_{1-\eta}}_*(\om)}^2\notag\\
	& \leq   \abs{{u_0}_*(\om)}^2+\abs{{u_1}_*(X_j)}^2-\eta_*(\om)f_*(X_j)+Q(\eta,X\eta),
\end{align}
where $Q(\eta,X\eta)$ is the quadratic term involving $\eta$ and its derivatives. 
\begin{rem}
	We actually assume the minimal property of the total energy. \textbf{What is the relation between the total energy and the directional energy?} In the classical case, the total energy is the average of the directional energy over all the directions, that is, the $n-1$ dimensional sphere $S^{n-1}$. Here, since the space is non-isotropic, this seems to be not easy to verify. However, we can still derive similar results in some special settings, for example, the Heisenberg group.  
\end{rem}

In the sequel, we assume our sub-riemannian manifold to be the Heisenberg group, i.e. $M=\H$. We claim that 
\begin{lem}\lab{lem5.3}
	For any $\psi\in C_0(\Om)$, we have
	\bel{inq4.22}
	\lim_{\eps\rightarrow0}\fint_{B_\eps(p)}\f{\eta(p\cdot\delta_\eps(w))-\eta(p)}\eps\f{f(p\cdot\delta_\eps(w))-f(p)}{\eps}d\mu(w)=\eta_*(X)f_*(X)+\eta_*(Y)f_*(Y),
	\qe
	where $\eta_*(X)f_*(X)=\sum_jX_j\eta X_jf$.
\end{lem}
To show this lemma, we need to prove the following two simple identities, where we employ the symmetries of the Heisenberg group, this is the only place that we use this property. 
\begin{lem}\lab{lem5.4}
	For any $p\in\H$ and any positive number $r$, we have 
	\bel{eq5.7}
	\int_{B_r(p)}w_iw_jd\mu=0, \quad \int_{B_r(p)}w_i^2d\mu=\int_{B_r(p)}w_j^2d\mu, \quad i\neq j,
	\qe
	where $w_i$ is the i-th component of the vector $w\in B_r(p)$, $i=1,\cdots,2n$.
\end{lem}
\begin{proof}[Proof of Lemma \ref{lem5.4}]
	The proof is a direct computation. By scaling and translation, we may assume $r=1$ and $p=0$. We can explicitly write down the geodesic in the unit ball $B$. Recall a general fact that for any $w=(x,y,z)\in B$, there exists a unique geodesic connecting $0$ and $w$, if $\abs{x}^2+\abs{y}^2\neq0$, otherwise, there are infinity many geodesics connecting $0$ and $w$. In other words, we may exclude the $z$-axis without changing the volume of the unit ball. In this case, we have for any $w_0=(x_0,y_0,z_0)$, with
	\[
	x_0^i=\f{A_i(\cos(\rho_0\phi)-1)+B_i\sin(\rho_0\phi)}\phi, y_0^i=\f{B_i(\cos(\rho_0\phi)-1)-A_i\sin(\rho_0\phi)}\phi,
	\]
	where $A_i^2+B_i^2=1, \rho_0=\rho(0,w_0),z_0\in\R$ and $\phi\in[-\f{2\pi}{\rho_0},\f{2\pi}{\rho_0}]$ is the solution of the equation
	\[
	\f{1-\cos(\rho_0\phi)}{\rho_0\phi-\sin(\rho_0\phi)}=\f{\rho_0^2}{z_0}.
	\]
	The unique geodesic connecting $0$ and $w_0$ can be parameterized by 
	\bel{eq4.24}
	\begin{gathered}
		x^i(s)=\f{A_i(\cos(s\rho_0\phi)-1)+B_i\sin(s\rho_0\phi)}\phi,\\
		y^i(s)=\f{B_i(\cos(s\rho_0\phi)-1)-A_i\sin(s\rho_0\phi)}\phi,\\
		z(s)=2\f{s\rho_0\phi-\sin(s\rho_0\phi)}{\phi^2}.
	\end{gathered}
	\qe
	where $z(1)=z_0$. To show this lemma, we define a map $\Phi:S^{2n-1}\times[-2\pi,2\pi]\times[0,1]\rightarrow B$ by setting $\Phi(A_i,B_i,\phi,s)=(x^i(s),y^i(s),z(s))$ where $(x^i,y^i,z)$ are given by \rf{eq4.24}. Since the geodesic is scaling invariant, we assume $s=1$. Thus, $w_i=x^i, w_{i+n}=y^i, i=1,\cdots,n$, we can also calculate the Jacobi of the transformation $\Phi$. Indeed, we have 
	\[
	\det(J\Phi)=2^{2n+2}\rho_0^2\left(\frac{\sin(\varphi/2)}{\varphi}\right)^{2n-1}\frac{\sin(\varphi/2)-(\varphi/2)\cos(\varphi/2)}{\varphi^{3}}.
	\]
	Since the Jacobi only involves the variable $\phi$ and we use the product metric of the space $S^{2n-1}\times[-2\pi,2\pi]$. Thus we exactly encounter the same setting as the Euclidean space. 
	
\end{proof}
We can now give the proof of Lemma \ref{lem5.3}.
\begin{proof}[Proof of Lemma \ref{lem5.3}]
	 
	We set $p=(p_1,\cdots,p_{2n},p_{2n+1}),w=(w_1,\cdots,w_{2n},w_{2n+1}))$, then 
	\[
	p\cdot\delta_\eps(w)=(p_i+\eps w_i,p_{2n+1}+\eps^2w_{2n+1}+\f12\sum_j^n\eps(p_iw_{i+n}-p_{i+n}w_i)).
	\]
	Using the Taylor expansion, we get 
	\begin{align*}
		f(p\cdot\delta_\eps(w))-f(p)=\eps\sum_{i=1}^{2n}\pa_{x_i}f w_i+\f12\eps\pa_{x_{2n+1}}f(p_iw_{i+n}-p_{i+n}w_i)+O(\eps^2).
	\end{align*}
	And similarly, we have 
	\[
	\eta(p\cdot\delta_\eps(w))-\eta(p)=\eps\sum_{i=1}^{2n}\pa_{x_i}\eta w_i+\f12\eps\pa_{x_{2n+1}}\eta(p_iw_{i+n}-p_{i+n}w_i)+O(\eps^2).
	\]
	Hence we finally obtain 
	\begin{align}\lab{eq5.8}
		& \f{\eta(p\cdot\delta_\eps(w))-\eta(p)}\eps\f{f(p\cdot\delta_\eps(w))-f(p)}{\eps}\notag\\
		= & \sum_{i=1}^{2n}\pa_{x_i}f\pa_{x_i}\eta w_i^2+\sum_{i\neq j}\pa_{x_i}f\pa_{x_j}\eta w_iw_j\notag\\
		+ & \f14\pa_{x_{2n+1}}f\pa_{x_{2n+1}}\eta\Big(\sum_{i=1}^n(p_iw_{i+n}-p_{i+n}w_i)\Big)^2\notag\\
		+ & \f12\pa_{x_{2n+1}}\eta\sum_{i=1}^{2n}\pa_{x_i}fw_i\sum_{j=1}^n(p_jw_{j+n}-p_{j+n}w_j)\notag\\
		+ & \f12\pa_{x_{2n+1}}f\sum_{i=1}^{2n}\pa_{x_i}\eta w_i\sum_{j=1}^n(p_jw_{j+n}-p_{j+n}w_j)
	\end{align}
	Integrate both sides of the identity \rf{eq5.8} and employ the identity \rf{eq5.7} in Lemma \ref{lem5.4}, we get the desired result. 
\end{proof}
\begin{rem}
The guiding principle of all the facts is simply that we view the approximating density\foo{The total measure $e_\eps d\mu$, the directional measure $e_\eps^jd\mu$, and even the signed Radon measure $d\nu=\f{\eta(y)-\eta(x)}\eps\f{f(y)-f(x)}{\eps}d\mu$.} as an approximating measure, and use the sub-partition lemma to show the convergence in an almost monotonous way.
\end{rem}
Combining the inequality \rf{inq4.22} and \rf{inq4.5}, we arrive at the following inequality 

\begin{align}\lab{inq4.25}
E^{u_\eta}(\psi)+E^{u_{1-\eta}}(\psi) & \leq E^{u_0}(\psi)+E^{u_1}(\psi)\notag\\
& - \int_\Om \psi X\eta Xf+\int_\Om \psi Q(\eta,X\eta).
\end{align}
We want to compare the energy of the corresponding maps using the minimal property of the maps $u_0$ and $u_1$. Thus we need to take supremum on both sides of the inequality \rf{inq4.25}, replace the test function $\eta$ by $t\cdot \eta$, and finally take the limit when $t\rightarrow0$. The second term in the second line of the inequality is negligible because it involves second order terms of the Lipschitz function $\eta$.  As a first step, we have 
\bel{inq4.26}
E^{u_\eta}+E^{u_{1-\eta}}\leq E^{u_0}+E^{u_1}-\sup_{\begin{subarray}{}0\leq \psi\leq 1 \\ \psi\in C_0(\Om)
\end{subarray}}\Big(\int_\Om \psi X\eta Xf - \int_\Om \psi Q(\eta,X\eta)\Big).
\qe
Using the minimizing property of the maps $u_0,u_1$, inequality \rf{inq4.26} implies
\[
\sup_{\begin{subarray}{}0\leq \psi\leq 1 \\ \psi\in C_0(\Om)
\end{subarray}}\Big(\int_\Om \psi X\eta Xf - \int_\Om \psi Q(\eta,X\eta)\Big)\leq0.
\]
Now replacing $\eta$ by $t\eta$ and letting $t\rightarrow0$, we obtain the following 
\[
\sup_{\begin{subarray}{}0\leq \psi\leq 1 \\ \psi\in C_0(\Om)
\end{subarray}}\int_\Om \psi X\eta Xf\leq0.
\]
Finally, we use an approximation argument. That is, denote 
\[
\Om_\eps=\set{x\in\Om|d_{cc}(x,\pa\Om)\geq\eps}.
\]
Choose some cut-off function $\psi_\eps$ satisfying 
\[
\psi_\eps=1\quad \text{in}\quad \Om_\eps,\quad \psi_\eps=0\quad \Om\setminus \Om_{\f\eps2}.
\]
\begin{rem}
Here of course we need the boundary $\pa\Om$ satisfying some certain condition, that is, we want the domain $\Om$ to shrink into a sub-domain such that there are some space between $\Om$ and $\Om_\eps$. In this way, we can construct the cut-off function $\psi_\eps$. Actually, we should assume that  the boundary is non-characteristic in the sense that for any point $x\in\Om$, there exists some vector $X_j(x)\notin T_x\pa\Om$. In this case, at least along the direction $X_j(x)$, the point x will leave the boundary towards the interior of $\Omega$.
\end{rem}
 
Therefore, we get a family of cut-off functions $\psi_\eps$ with compact support in $\Om$. A simple estimate gives that 
\begin{align}\lab{inq1.39}
\abs{\int_\Om\psi_\eps X\eta Xf-\int_\Om X\eta Xf} & \leq \int_{\Om\setminus \Om_\eps}\abs{1-\psi_\eps}\abs{X\eta}\abs{Xf}\notag\\
 & \leq \int_{\Om\setminus\Om_\eps}\abs{X\eta}\abs{Xf}.
\end{align}
We see that \rf{inq1.39} approaches zero, as the sub-domain $\Om_\eps$ gets closer to $\Om$ with $\eps\rightarrow0$.
Hence we conclude  
\bel{inq4.28}
\int_\Om X\eta Xf\leq 0,
\qe
which is the weak subsolution of the equation 
\[
Hf=0,
\]
where $H=-X^*X$ is the H\"ormander operator. Inequality \rf{inq4.28} is actually a weak form of the Bochner formula. 
\subsection{The main theorem}
In this section, we will present the proof of the main theorem whose statement was written in the introduction. Using the notation established throughout the paper, we re-state the theorem in more precise terms as follows.
\begin{thm}
Any energy minimising harmonic map $u:\H \longrightarrow N$, from the Heisenberg group to a non-positively curved space $N$ ($\operatorname{CAT}(0)$) is locally Lipschitz continuous. Moreover, we have the estimate 
\[
Lip_u(x)\leq \f C{\mu(B_r(x))}E^u
\]
where $C=C(n,N,r_0)$ is a positive constant depending only on the geometry of $N$. 
\end{thm}
\begin{proof}
We now specify the maps: $u_0(x)=u(x)$, $u_1(x)=u(x+\eps\om_j)=u_j(x)$, and $\om_j\in V(x)$ in the inequality \rf{inq4.28}. We then conclude that the difference function $f(x)=d^2(u(x),u(x+\om_j))$ is a weak sub-solution and hence by the De Giorgi-Nash-Moser iteration, 
there exists some positive constant $C,r_0$, such that
\bel{inq5.1}
d^2(u(x),u(x+\eps\om_j))\leq C\fint_{B_r(x)}d^2(u(y),u(y+\eps\om_j))d\mu(y)
\qe
for any $0<r<r_0$ and $x\in\Om, dist(x,\pa\Om)>r_0$. We may rewrite inequality \rf{inq5.1} as
\begin{align*}
d^2(u(x),u(x+\eps\om_j))\leq & C\eps^2 \fint_{B_r(x)}\f{d^2(u(y),u(y+\eps\om_j))}{\eps^2}d\mu(y)\\
\leq & C\eps^2\f1{\mu(B_r(x))}E_\eps^u\\
\leq & C\eps^2\f1{\mu(B_{r_0}(x))}E_\eps^u.
\end{align*}
But the volume of the ball is of polynomial growth, so we can compare the volumes of of $B_r(x)$ and $B_{r_0}(x)$. Dividing both sides by $\eps^2$ and letting $\eps\rightarrow0$, we then conclude that the map $u$ is Lipschitz continuous and 
\[
Lip_u(x)\leq  CE^u,
\]
giving the desired estimate.
\end{proof}

\bibliographystyle{plain}
\bibliography{refoflip}

\date{November 2021}

\end{document}